\RequirePackage[l2tabu, orthodox]{nag}
\documentclass[10pt]{amsart}
\usepackage[pdftex, pdftitle={Another Christoffel--Darboux Formula  for Multiple Orthogonal Polynomials of Mixed Type}, pdfsubject={Alternative CD formula for MOPRL of Mixed Type},pdfauthor={Gerardo Araznibarreta, Manuel Manas}]{hyperref}
\usepackage[pdftex]{color}
\usepackage[english]{babel}
\usepackage{mathtools}
\usepackage{amssymb,amsthm,amscd}

\usepackage[latin1]{inputenc}
\usepackage{mathrsfs}
\usepackage[thinlines]{easybmat}

\newcommand{\C}{\mathbb{C}}

\newcommand{\R}{\mathbb{R}}

\newcommand{\Z}{\mathbb{Z}}
\newcommand{\N}{\mathbb{N}}

\newcommand{\A}{\mathscr A}

\renewcommand{\d}{\mathrm{d}}

\newcommand{\Exp}[1]{\operatorname{e}^{#1}}

\newcommand{\rI}{\operatorname{I}}
\newcommand{\rII}{\operatorname{II}}

\newcommand{\Cs}{\mathscr C}

\newtheorem{pro}{Proposition}
\newtheorem{lemma}{Lemma}
\newtheorem{definition}{Definition}
\newtheorem{theorem}{Theorem}

\begin{document}

\title[Alternative CD formula for MOPRL of Mixed Type]{ Another Christoffel--Darboux Formula  for Multiple Orthogonal Polynomials of Mixed Type}
\author{Gerardo Araznibarreta}
\address{Departamento de F\'{\i}sica Te\'{o}rica II (M\'{e}todos Matem\'{a}ticos de la F\'{\i}sica), Universidad Complutense de Madrid, 28040-Madrid, Spain}
\email{gariznab@ucm.es}
\author{Manuel Ma\~{n}as}
\email{manuel.manas@ucm.es}
\maketitle

\begin{abstract}
An alternative expression for the Christoffel--Darboux formula  for multiple orthogonal polynomials of mixed type is derived from the $LU$ factorization of the
moment matrix of a given measure and two sets of weights. We use the action of the generalized Jacobi matrix $J$, also responsible for the recurrence relations, on the linear forms and their duals to obtain the  result.
\end{abstract}

\section{Introduction}

In this paper we address a natural question that arises from the $LU$ factorization approach to multiple orthogonality  \cite{afm}.  The Gauss--Borel factorization of a Hankel matrix, which plays the role of a moment matrix, leads in the classical case to a natural description of algebraic facts regarding orthogonal polynomials on the real line (OPRL) such as recursion relations and Christoffel--Darboux formula. In that case we have a chain of orthogonal polynomials $\{P_l(x)\}_{l=0}^\infty$ of increasing degree $l$. In \cite{afm} we extended that approach to the multiple orthogonality scenario, and the Gauss--Borel factorization of an appropriate moment matrix leaded to sequences of families of multiple orthogonal polynomials in the real line (MOPRL), $\big\{Q_{[\vec\nu_{1}(l);\vec\nu_{2}(l-1)]}^{(\rII,a_1(l))}\big\}_{l=0}^\infty$ and $\big\{\bar Q^{(\rI,a_1(l))}_{[\vec\nu_{2}(l);\vec\nu_{1}(l-1)]}\big\}_{l=0}^\infty$. 
The recursion relations are relations constructed in terms of the linear forms in these sequences. However, the Christoffel--Darboux formula given  in Proposition \ref{pro:cd}, that  was re-deduced by linear algebraic means (Gauss--Borel factorization), was not expressed in terms of linear forms belonging to the mentioned sequences. This situation is rather different to the OPRL case, in that classical case the Christoffel-Darboux formula is expressed in terms of orthogonal polynomials in the sequence. The aim of this paper is to show that, within that scheme, we can deduce  an alternative but equivalent MOPRL Christoffel--Darboux formula constructed in terms of linear forms in the sequences $\big\{Q_{[\vec\nu_{1}(l);\vec\nu_{2}(l-1)]}^{(\rII,a_1(l))}\big\}_{l=0}^\infty$ and $\big\{\bar Q^{(\rI,a_1(l))}_{[\vec\nu_{2}(l);\vec\nu_{1}(l-1)]}\big\}_{l=0}^\infty$ as in OPRL situation.

\subsection{Historical background}
Simultaneous rational approximation starts back in 1873 when Hermite proved the transcendence of the Euler number e \cite{He}.  Later,  K. Mahler delivered  at the University of Groningen several lectures \cite{Ma} where he settled down the  foundations of this theory, see also \cite{Co} and \cite{Ja}.  Simultaneous rational approximation when expressed in terms of Cauchy transforms leads to multiple orthogonality of polynomials. Given an interval $\Delta \subset \R$ of the real line, let ${\mathcal{M}}(\Delta)$ denote all the finite positive Borel measures with support containing infinitely many points in $\Delta$. Fix $\mu \in {\mathcal{M}}(\Delta)$, and let us consider  a system of weights $\vec w=(w_1,\ldots,w_p)$ on $\Delta$, with $p \in {\mathbb{N}}$; i.e. $w_1,\dots,w_p$ being  real integrable functions on $\Delta$ which does not change sign on $\Delta$. Fix a multi-index $\vec \nu=(\nu_1,\ldots, \nu_p) \in {\mathbb{Z}}_+^p,$ ${\mathbb{Z}}_+=\{0,1,2,\ldots\}$, and denote $|\vec \nu|=\nu_1+\cdots+\nu_p$.
Then, there exist polynomials, $A_1,\ldots, A_p$  not all identically equal to zero which satisfy the following orthogonality relations
\begin{align}\label{tipoI}
\int_{\Delta}  x^{j} \sum_{a=1}^{p} A_a(x)w_{a} (x)\d\mu (x)&=0,  & \deg
 A_{a}&\leq\nu_{a}-1,&   j&=0,\ldots, |\vec \nu|-2.
\end{align}
Analogously, there exists a polynomial $B$ not identically equal to zero, such that
\begin{align}\label{tipoII}
\int_{\Delta} x^{j} B (x) w_{b} (x) \d\mu (x)&=0, & \deg B &\leq|\vec \nu|,&  j&=0,\ldots, \nu_b-1, \quad b=1,\ldots,p.
\end{align}
These are the so called multiple orthogonal polynomials of type I and type II, respectively, with respect to the combination $(\mu, \vec w, \vec \nu)$ of the measure $\mu$, the systems of weights $\vec w$ and the multi-index $\vec \nu.$ When $p=1$ both definitions coincide with standard orthogonal polynomials on the real line. Given a measure $\mu \in {\mathcal{M}}(\Delta)$ and a system of weights $\vec w$ on $\Delta$ a multi-index $\vec \nu$ is called type I or type II normal if $\deg A_a$ must equal to $\nu_a-1,$ $a=1,\ldots,p$, or $\deg B$ must equal to $|\vec \nu|-1$, respectively. When for a pair $(\mu, \vec w)$ all the multi-indices are type I or type II normal, then the pair is called type I perfect or type II perfect respectively. Multiple orthogonal of polynomials have been employed in several proofs of irrationality of numbers. For example, in \cite{Be} F. Beukers shows that Apery's proof  \cite{Ap} of the irrationality of $\zeta(3)$ can be placed in the context of a combination of type I and type
II multiple orthogonality which is called mixed type multiple orthogonality of polynomials. More recently, mixed type approximation has appeared in random matrix and non-intersecting Brownian motion theories, \cite{BleKui}, \cite{daems-kuijlaars}, \cite{arno}.  Sorokin \cite{So} studied a simultaneous rational approximation construction which is closely connected with multiple orthogonal  polynomials of mixed type. In \cite{assche}  a Riemann--Hilbert problem was found that characterizes multiple orthogonality of type I and II, extending in this way the result previously found in \cite{fokas} for standard orthogonality. In \cite{daems-kuijlaars} mixed type multiple orthogonality was analyzed from this perspective.
\subsection{Perfect systems and MOPRL of mixed type}
In order to introduce  multiple orthogonal polynomials of mixed type we consider two systems of weights  $\vec w_1=(w_{1,1},\dots,w_{1,p_1})$  and $\vec w_2=(w_{2,1},\dots,w_{2,p_2})$ where  $p_1,p_2\in\N$,  and two multi-indices $\vec \nu_1=(\nu_{1,1},\dots,\nu_{1,p_1})\in{\mathbb{Z}}_+^{p_1}$ and $\vec \nu_2=(\nu_{2,1},\dots,\nu_{2,p_2})\in {\mathbb{Z}}_+^{p_2}$ with $|\vec \nu_1|=|\vec \nu_2|+1$. There exist polynomials $A_1,\ldots, A_{p_1},$ not all identically zero, such that $\deg A_s < \nu_{1,s}$, which satisfy the following relations
\begin{equation}\label{orth}
\int_{\Delta} \sum_{a=1}^{p_1} A_a(x)w_{1,a} (x)w_{2,b}(x)x^{j} d\mu (x)=0, \quad  j=0,\ldots, \nu_{2,b}-1,\quad b=1,\ldots,p_2.
\end{equation}
They are called mixed multiple-orthogonal polynomials with respect to the combination $(\mu,\vec w_1,\vec w_2,\vec \nu_1,\vec\nu_2)$ of the measure $\mu,$ the systems of weights $\vec w_{1}$ and $\vec w_2$ and the multi-indices $\vec \nu_1$ and $\vec \nu_2.$  It is easy to show that finding the polynomials $A_1,\ldots, A_{p_1}$ is equivalent to solving  a system of $|\vec \nu_2|$ homogeneous linear equations  for the $|\vec \nu_1|$ unknown coefficients of the polynomials. Since $|\vec \nu_1|=|\vec \nu_2|+1$  the system always has a nontrivial solution.   The matrix of this system of equations is the so called moment matrix, and the study of its Gauss--Borel factorization will be the cornerstone of this paper. Observe that when $p_1=1$ we are in the type II case and if $p_2=1$ in  type I case. Hence in general we can find a solution of  \eqref{orth} where there is an $a \in \{1,\ldots,p_1\}$ such that $\deg A_a< \nu_{1,a}-1$. When given a combination $(\mu,\vec w_1,\vec w_2)$ of a measure $\mu \in {\mathcal{M}
}(\Delta)$ and systems of weights $\vec w_1$ and $\vec w_2$ on $\Delta$ if for each pair of multi-indices $(\vec \nu_1,\vec \nu_2)$ the conditions \eqref{orth} determine that $\deg A_a=\nu_{1,a}-1,$ $a=1, \ldots, p_1$, then we say that the combination $(\mu,\vec w_1,\vec w_2)$ is perfect. In this case
we can determine a unique system of mixed type orthogonal polynomials $\big(A_1, \ldots, A_{p_1}\big)$ satisfying \eqref{orth} requiring for  $a_1 \in \{1, \ldots p_1\}$ that $A_{a_1}$ monic. Following \cite{daems-kuijlaars} we say that we have a type II normalization and denote the corresponding system of polynomials by   $A_a^{(\rII,a_1)},$ $j=1, \ldots, p_1$. Alternatively, we can proceed as follows, since the system of weights is perfect from \eqref{orth} we deduce that
\begin{align*}
\int x^{\nu_{1,r_1}} \sum_{a=1}^{p_1} A_a(x)w_{1,a} (x)w_{2,b}(x) \d\mu (x)\neq 0.
\end{align*}
Then, we can determine a unique system of mixed type of multi-orthogonal polynomials
$(A_1^{(\rI,a_2)},\ldots,A_{p_2}^{(\rI,a_2)})$ imposing that
\begin{align*}
\int x^{\nu_{1,a_2}} \sum_{a=1}^{p_1} A_a^{(\rI,a_2)}(x)w_{1,a} (x)w_{2,b}(x) \d\mu (x)=1,
\end{align*}
which is a type I normalization. We will use the notation $A_{[\vec\nu_1;\vec\nu_2],a}^{(\rII,a_1)}$ and $A_{[\vec\nu_1;\vec\nu_2],a}^{(\rI,a_2)}$ to denote these multiple orthogonal polynomials with type II and I normalizations, respectively.

A known illustration of perfect combinations $(\mu, \vec w_1, \vec w_2)$ can be constructed with an arbitrary positive finite Borel measure $\mu$ and  systems of weights formed with exponentials:
\begin{align}\label{exponentials}
&  (\Exp{\gamma_1x},\ldots,\Exp{\gamma_px}),& \gamma_i &\neq \gamma_j, &i &\neq j,& i,j& = 1,\ldots,p,
\end{align}
 or by binomial functions
\begin{align}\label{binomials}
  &((1-z)^{\alpha_1},\ldots,(1-z)^{\alpha_p}),& \alpha_i -\alpha_j &\not\in {\mathbb{Z}}, &i &\neq j,& i,j& =
  1,\ldots,p.
\end{align}
or combining both classes, see \cite{Nik}. Recently, in \cite{LF2} the authors were able to prove perfectness for a wide class of systems of weights. These systems of functions, now called Nikishin systems, were introduced by E.M. Nikishin  \cite{Nik} and initially named MT-systems (after Markov and Tchebycheff).

\subsection{Borel--Gauss factorization and multiple orthogonality of mixed type. A remainder}

Orthogonal polynomials and the theory of integrable systems has been connected in several ways in the mathematical literature. We are particularly interested in the one based in the Gauss--Borel factorization that was developed in \cite{adler}-\cite{adler-van moerbeke 2}, and applied further in \cite{cum}-\cite{afm2}.  These papers set the basis for the method we use in this paper to get an alternative Christoffel--Darboux formula for MOPRL of mixed type.

In the following we extract from  \cite{afm} the necessary material for the construction of the mentioned alternative Christoffel--Darboux formula. We introduce the moment matrix and recall how the Borel--Gauss factorization leads to multiple orthogonality. Then, we outline how the recursion relations appears by introducing a Jacobi type semi-infinite  matrix and recall the reader the Chistoffel--Darboux formula \cite{daems-kuijlaars0,daems-kuijlaars}.
\subsubsection{The moment matrix}
Fixed a composition $\vec n_\alpha$, $\alpha=1,2$, any given $l\in\Z_+:=\{0,1,2,\dots\}$, see \cite{stanley},  determines uniquely the following non-negative
integers $k_{\alpha}(l)\in\Z_+$, $a_{\alpha}(l)\in \{1,2,\dots,p_{\alpha}\}$ and $r_{\alpha}(l)$ such that
$0\leq r_{\alpha}(l)<n_{\alpha,a_{\alpha}(l)}$ and
\begin{align}\label{i}
l&=\begin{cases}
      k_{\alpha}(l)|\vec n_{\alpha}|+n_{\alpha,1}+\dots+n_{\alpha,a_{\alpha}(l)-1}+r_{\alpha}(l), &  a_{\alpha}(l)\neq1,\\
      k_{\alpha}(l)|\vec n_{\alpha}|+r_{\alpha}(l),  & a_{\alpha}(l)=1.
  \end{cases}
\end{align}

We define now the monomial strings as vectors that may be understood as sequences of monomials according to the
composition $\vec n_{\alpha}$, $\alpha=1,2$,  introduced previously.
\begin{align*}
 \chi_{\alpha}&:=\begin{pmatrix}
                 \chi_{\alpha,[0]}\\
                 \chi_{\alpha,[1]}\\
                 \vdots\\
                 \chi_{\alpha,[k]}\\
                 \vdots
                \end{pmatrix} &\mbox{where}& &
 \chi_{\alpha,[k]}:=\begin{pmatrix}
                 \chi_{\alpha,[k],1}\\
                 \chi_{\alpha,[k],2}\\
                 \vdots\\
                 \chi_{\alpha,[k],a_{\alpha}}\\
                 \vdots\\
                 \chi_{\alpha,[k],p_{\alpha}}
                    \end{pmatrix} & &\mbox{and}& &
\chi_{\alpha,[k],a_{\alpha}}:=\begin{pmatrix}
                      x^{kn_{\alpha,a_{\alpha}}}\\
                      x^{kn_{\alpha,a_{\alpha}}+1}\\
                      \vdots\\
                      x^{kn_{\alpha,a_{\alpha}}+(n_{\alpha,a_{\alpha}}-1)}
                     \end{pmatrix}.
\end{align*}
In a similar manner for  $\alpha=1,2$ we  define the weighted monomial strings
\begin{align*}
\xi_{\alpha}&:=\begin{pmatrix}
                 \xi_{\alpha,[0]}\\
                 \xi_{\alpha,[1]}\\
                 \vdots\\
                 \xi_{\alpha,[k]}\\
                 \vdots
                \end{pmatrix} &\mbox{where}& &
 \xi_{\alpha,[k]}:=\begin{pmatrix}
                 w_{\alpha,1}\chi_{\alpha,[k],1}\\
                 w_{\alpha,2}\chi_{\alpha,[k],2}\\
                 \vdots\\
                 w_{\alpha,p_{\alpha}}\chi_{\alpha,[k],p_{\alpha}}
                    \end{pmatrix}.
\end{align*}
For any given $l\in\Z_+$ and $a_{\alpha}:={1,2,\dots,p_{\alpha}}$ we define
\begin{align*}
 \nu_{\alpha,a_{\alpha}}(l):=\begin{cases}
                              k_{\alpha}(l)|\vec n_{\alpha}|+n_{\alpha,a_{\alpha}}-1, & a_{\alpha}<a_{\alpha}(l),\\
                              k_{\alpha}(l)|\vec n_{\alpha}|+r_{\alpha}(l), & a_{\alpha}=a_{\alpha}(l),\\
                              k_{\alpha}(l)|\vec n_{\alpha}|-1,& a_{\alpha}>a_{\alpha}(l).
                             \end{cases}
\end{align*}
Notice  that $\nu_{\alpha,a_{\alpha}}(l)$ is the hightest degree of all the monomials
of type $a_{\alpha}$ up to the component $\chi_{\alpha}^{(l)}$ included, of the monomial string. Actually
\begin{align*}
\chi_{\alpha}^{(l)}=x^{\nu_{\alpha,a_{\alpha}(l)}}.
\end{align*}

 Given $l\geq 1$ and  $a_{\alpha}=1,\cdots, p_{\alpha}$ the $+$ ($-$) associated integer is  the smallest (largest) integer
$ l_{\{+,a_{\alpha}\}}$ ($l_{\{-,a_{\alpha}\}}$) such that  $l_{\{+,a_{\alpha}\}} \geq l$ ($l_{\{-,a_{\alpha}\}} \leq l$ )
and $a( l_{\{+,a_{\alpha}\}})=a_{\alpha}$ ($a( l_{\{-,a_{\alpha}}\})=a_{\alpha}$).
It can be shown that
\begin{align}\label{ai}
\begin{aligned}
  l_{\{-,a_{\alpha}\}}&:=\begin{cases} k_{\alpha}(l)|\vec n_{\alpha}|+\sum_{i=1}^{a_{\alpha}}n_{\alpha,i}-1, & a_{\alpha}<a_{\alpha}(l), \\
  l, &a_{\alpha}=a_{\alpha}(l),\\
  k_{\alpha}(l)|\vec n_{\alpha}|-\sum_{i=a_{\alpha}+1}^{p_{\alpha}}n_{\alpha,i}-1, & a_{\alpha}>a_{\alpha}(l-1), \end{cases}\\
  l_{\{+,a_{\alpha}\}}&:=\begin{cases}
    (  k_{\alpha}(l)+1)|\vec n_{\alpha}|+\sum_{i=1}^{a_{\alpha}-1}n_{\alpha,i},& a_{\alpha}<a_{\alpha}(l),\\
  l,&a_{\alpha}=a_{\alpha}(l),\\
 ( k_{\alpha}(l)+1)|\vec n_{\alpha}|-\sum_{i=a_{\alpha}}^{p_{\alpha}}n_{\alpha,i},& a_{\alpha}>a_{\alpha}(l).
 \end{cases}
 \end{aligned}
\end{align}

Finally given the weighted monomials $\xi_{\vec n_{\alpha}}$, associated to the compositions $\vec n_{\alpha}$,  $\alpha=1,2$, we introduce
the moment matrix in the following manner
\begin{definition}
  The moment matrix is given by
\begin{align}
  \label{compact.g}
  g_{\vec n_1,\vec n_2}:=\int \xi_{\vec n_1}(x)\xi_{\vec n_2}(x)^\top\d \mu(x).
\end{align}
\end{definition}
\subsubsection{Multiple Orthogonality of mixed type}

\begin{definition}\label{ladef}
For a given a perfect combination  $(\mu, \vec w_1,\vec w_2)$ we define
\begin{enumerate}
  \item The  Gauss--Borel factorization  (also known as   $LU$  factorization) of a semi-infinite moment matrix $g$,
  determined by $(\mu, \vec w_1,\vec w_2)$,  is the
  problem of finding the solution of
\begin{align}\label{facto}
  g&=S^{-1}\bar S, & S&=\begin{pmatrix}
    1&0&0&\cdots \\
    S_{1,0}&1&0&\cdots\\
    S_{2,0}&S_{2,1}&1&\cdots\\
    \vdots&\vdots&\vdots&\ddots
    \end{pmatrix}\in G_{-}, &
   \bar S^{-1}&=\begin{pmatrix}
     \bar S_{0,0}'&\bar S_{0,1}'&\bar S_{0,2}'&\cdots\\
     0&\bar S_{1,1}'&\bar S_{1,2}'&\cdots\\
     0&0&\bar S_{2,2}'&\cdots&\\
     \vdots&\vdots&\vdots&\ddots
   \end{pmatrix}\in G_{+},
\end{align}
where $S_{i,j},\bar S'_{i,j}\in\R.$
\item In terms of these matrices we construct the polynomials
    \begin{align} \label{defmops}
A^{(l)}_a:={\sum}'_iS_{l,i}x^{k_1(i)},
\end{align}
where the sum $\sum'$ is taken for a fixed $a=1,\dots,p_1$ over those $i$
such that $a=a_1(i)$ and $i\leq l$. We also construct the dual polynomials
\begin{align} \label{defdualmops}
\bar A^{(l)}_b:={\sum}'_jx^{k_2(j)}\bar S_{j,l}',
\end{align}
where the sum $\sum'$ is taken for a given $b$ over those $j$ such that $b=a_2(j)$ and $j\leq l$.
\item   Strings of linear forms and dual linear forms associated with  multiple ortogonal polynomials and their duals  are
  defined by
  \begin{align}\label{linear forms S}
Q:=  \begin{pmatrix}
    Q^{(0)}\\
    Q^{(1)}\\
    \vdots
  \end{pmatrix}&=S\xi_{1},&
  \bar Q:=\begin{pmatrix}
    \bar Q^{(0)}\\
    \bar Q^{(1)}\\
    \vdots
  \end{pmatrix}&=(\bar S^{-1})^\top\xi_{2},
\end{align}
\end{enumerate}

\end{definition}
%
%
%
%
%
Then
\begin{pro}
  \label{proposition: linear forms and mops}
  \begin{enumerate}
    \item The linear  forms and their duals, introduced in Definition \ref{ladef}, are given by
\begin{align}\label{linear.forms}
 Q^{(l)}(x)&:= \sum_{a=1}^{p_1}A^{(l)}_{a}(x)w_{1,a}(x),&
\bar Q^{(l)}(x)&:= \sum_{b=1}^{p_2}\bar  A^{(l)}_{b}
(x)w_{2,b}(x).
\end{align}
\item The orthogonality relations
 \begin{align}\label{linear.form.orthogonality}
 \begin{aligned}
  \int  Q^{(l)}(x)w_{2,b}(x)x^k\d \mu(x)&=0,&0&\leq k\leq \nu_{2,b}(l-1)-1,&b&=1,\dots,p_2,\\
      \int  \bar Q^{(l)}(x)w_{1,a}(x)x^k\d \mu(x)&=0,&0&\leq k\leq \nu_{1,a}(l-1)-1,&a&=1,\dots,p_1,
 \end{aligned}
   \end{align}
   are fulfilled.
   \item   The following multiple bi-orthogonality relations among linear forms and their duals
\begin{align}\label{biotrhoganility}
  \int  Q^{(l)}(x)\bar Q^{(k)}(x)\d \mu(x)&=\delta_{l,k},& l,k\geq 0,
\end{align}
hold.
\item     We have the following identifications
    \begin{align*}
A^{(l)}_a&=A_{[\vec\nu_{1}(l);\vec\nu_{2}(l-1)],a}^{(\rII,a_1(l))}, &
\bar A^{(l)}_b&= A^{(\rI,a_1(l))}_{[\vec\nu_{2}(l);\vec\nu_{1}(l-1)],b},
\end{align*}
in terms of multiple orthogonal polynomials of mixed type with two  normalizations $\rI$ and $\rII$, respectively.
  \end{enumerate}
\end{pro}
%
%
%
%

\subsubsection{Functions of the second kind}
The Cauchy transforms of the linear forms \eqref{linear.forms} play a crucial role in the
Riemann--Hilbert problem associated with the multiple orthogonal polynomials of mixed type
\cite{daems-kuijlaars}.  Observe that the construction of multiple orthogonal polynomials  performed so far is synthesized in the following
strings of multiple orthogonal polynomials and their duals
\begin{align}\label{mop-s}
\begin{aligned}
\A_{a}&:=  \begin{pmatrix}
    A^{(0)}_{a}\\
    A^{(1)}_{a}\\
    \vdots
  \end{pmatrix}=S\chi_{1,a},&
  \bar\A_{b}&:= \begin{pmatrix}
 \bar A^{(0)}_{b}\\
    \bar A^{(1)}_{b}\\
    \vdots
  \end{pmatrix}=(\bar S^{-1})^\top\chi_{2,b},&
\end{aligned}
\end{align}
for $a=1,\dots,p_1$ and $b=1,\dots,p_2$.
In order to complete these formulae and in terms of $  \chi^*_a:=z^{-1}\chi_a(z^{-1})$
let us introduce the  following formal semi-infinite vectors
\begin{align}\label{cauchy-S}
\begin{aligned}
  \Cs_b&=\begin{pmatrix}
    C_b^{(0)}\\C_b^{(1)}\\\vdots
  \end{pmatrix}=\bar S\chi_{2,b}^*(z),&
  \bar\Cs_a&=\begin{pmatrix}
    \bar C_a^{(0)}\\\bar C_a^{(1)}\\\vdots
  \end{pmatrix}=(S^{-1})^\top\chi_{1,a}^*(z),&
\end{aligned}\end{align}
for $a=1,\dots,p_1$ and $b=1,\dots,p_2$, that we call strings of second kind functions.
These objects are actually Cauchy transforms of the linear forms $Q^{(l)}$, $l \in {\mathbb{Z}}_+$, whenever the
series converge and outside the support of the measures involved.%
\begin{pro}\label{pro:cauchy tr}
For each  $l\in {\mathbb{Z}}_+$ the second kind functions can be expressed as follows
   \begin{align}
   \begin{aligned}
C_b^{(l)}(z)&=\int_\R\frac{Q^{(l)}(x)w_{2,b}(x)}{z-x}\d \mu(x),& z \in D_b^{(l)} \setminus \operatorname{supp}
(w_{1,b}\d \mu(x)),\\
\bar C_a^{(l)}(z)&=\int_\R \frac{\bar Q^{(l)}(x)w_{1,a}(x)}{z-x}\d \mu(x),&  z \in \bar D_a^{(l)} \setminus
\operatorname{supp} (w_{2,a}\d \mu(x)).
\end{aligned}
\end{align}
\end{pro}
\subsubsection{Recursion relations, a Jacobi type matrix}
The moment matrix has a  Hankel type symmetry that implies the recursion relations and the Christoffel--Darboux
formula. We  consider the shift operators $\Upsilon_{\alpha}$ defined by
\begin{align}
(\Upsilon_{\alpha})_{l,j}&:=\delta_{j,(l+1)_{\{+,a_{\alpha}(l)\}}}
\end{align}
Wich satisfy the following relation
\begin{align*}
 \Upsilon_{\alpha}\chi_{\alpha}(x)&=x\chi_{\alpha}(x) \Longrightarrow \Upsilon_{\alpha}\xi_{\alpha}(x)=x\xi_{\alpha}(x)
\end{align*}

In terms of these shift matrices we can describe the particular Hankel symmetries for the moment matrix
\begin{pro}
  The moment matrix $g$ satisfies the Hankel type symmetry
      \begin{align}\label{sym2}
 \Upsilon_1g  =
 g\Upsilon_2^\top.
  \end{align}
\end{pro}
From this symmetry we see that the following is consistent
\begin{definition}
We define the matrices
\begin{align*}
 J&:=S\Upsilon_1 S^{-1}=\bar S \Upsilon_2^\top\bar S^{-1}=J_++J_-,& J_+&:=(  S\Upsilon_1 S^{-1})_+, & J_-&:=(\bar S \Upsilon_2^\top\bar S^{-1})_-,
\end{align*}
where the sub-indices + and $-$denote the upper triangular and strictly lower triangular projections.
 \end{definition}
The recursion  relations follow immediately from the eigenvalue property
\begin{align}\label{rel}
 J Q(x)&=x Q(x) &  \bar{Q}(x)^{\top} J&=x \bar{Q}(x)^{\top}.
\end{align}

\subsubsection{Christoffel--Darboux formula}
\begin{definition}
The  Christoffel--Darboux kernel is
\begin{align}
  \label{def.CD}
  K^{[l]}(x,y)&:=\sum_{k=0}^{l-1}Q^{(k)}(y)\bar Q^{(k)}(x)
\end{align}
\end{definition}
In \cite{daems-kuijlaars0,daems-kuijlaars} it was shown using a Riemann--Hilbert problem approach that
\begin{pro}\label{pro:cd}
 For  $l\geq \max(|\vec n_1|,|\vec n_2|)$ the following Christoffel--Darboux formula
  \begin{align}
  \begin{aligned}
 (x-y)K^{[l]}(x,y)=&
\sum_{b =1}^{p_2}\bar Q^{(\rII,b)}_{[\vec\nu_2(l-1)+\vec e_{2,b};\vec \nu_1(l-1)]}(x)
Q^{(\rI,b)}_{[\vec\nu_1(l-1);\vec \nu_2(l-1)-\vec e_{2,b}]}(y)\\
&-\sum_{a=1}^{p_1}\bar Q^{(\rI,a)}_{[\vec\nu_2(l-1);\vec \nu_1(l-1)-\vec
e_{1,a}]}(x)Q^{(\rII,a)}_{[\vec\nu_1(l-1)+\vec e_{1,a};\vec \nu_2(l-1)]}(y).
\end{aligned}
    \label{cd3}
\end{align}
holds.
\end{pro}
Here $\{\vec e_{i,a}\}_{a=1}^{p_i}\subset \R^{p_i}$ stands for the vectors in the respective canonical basis, $i=1,2$.
In \cite{afm} it was given an algebraic proof of this statement not relying on analytic conditions. We refer the interested reader to \cite{simon} for a complete survey of the subject.

\section{Alternative Christoffel--Darboux formula for multiple orthogonal polynomials of mixed type}

The result of this paper is the following

\begin{theorem}
For $l\geq \max\{|\vec n_1|,|\vec n_2|\}$ the following Christofel--Darboux formula holds
  \begin{align*}
 (y-x)K^{[l]}(x,y)&=\smashoperator{\sum_{(i,j)\in \sigma_1[l]}}\bar{Q}(x)^{(j)} J_{j,i} Q(y)^{(i)}-
 \smashoperator{\sum_{(i,j)\in\sigma_2[l]}}\bar{Q}(x)^{(j)} J_{j,i} Q(y)^{(i)}
\end{align*}
where
\begin{align*}
  \sigma_1[l]&:=\big\{l,\dots, (l)_{\{+,r_1(a_1(l)-1)\}}\big\}\times \big\{(l-1)_{\{-,r_1(a_1(l-1)+1)\}},\dots,l-1\big\},\\
    \sigma_2[l]&:=\big\{(l-1)_{\{-,r_2(a_2(l-1)+1)\}},\dots,l-1\big\}\times\big\{l,\dots,(l)_{\{+,r_2(a_2(l)-1)\}}\big\}.
\end{align*}
\end{theorem}
\begin{proof}
Splitting the eigenvalue property  \eqref{rel} into blocks we get
\begin{align*}
 J Q(y)&=y Q(y)\Longrightarrow J^{[l]} Q(y)^{[l]}+J^{[l,\geq l]} Q(y)^{[\geq l]}=yQ(y)^{[l]} \\
 \bar{Q}(x)^{\top} J&=x \bar{Q}(x)^{\top}\Longrightarrow  [\bar{Q}(x)^{\top}]^{[l]} J^{[l]}+[\bar{Q}(x)^{\top}]^{[\geq l]} J^{[\geq l,l]}=x[\bar{Q}(x)^{\top}]^{[l]}
\end{align*}
Multiply the first equation from the left by $[\bar{Q}(x)^{\top}]^{[l]}$ and the second one from the right by
$Q(y)^{[l]}$ substract both results to obtain
\begin{align*}
 [\bar{Q}(x)^{\top}]^{[l]}J^{[l,\geq l]}Q(y)^{[\geq l]}-[\bar{Q}(x)^{\top}]^{[\geq l]} J^{[\geq l,l]}Q(y)^{[l]}&=
 (y-x)[\bar{Q}(x)^{\top}]^{[l]}\cdot Q(y)^{[l]}\\&=(y-x)K^{[l]}(x,y)
\end{align*}
After an brief  study of the shape of $J$ we realize that even though $J^{[l,\geq l]}$ has semi-infinte
length rows, most of its elements are 0. Actually it only contains a finite number of nonzero entries that concentrate
in the lower left corner of itself. The same reasoning applies to $J^{[\geq l,l]}$. This matrix has semi infinite
length columns but again it only contains a finite number of nonzero terms concentrated in the upper right corner of itself.
Of course the number of terms involved in this expression will depend on the value of $[l]$. To be more precise we proceed as follows.  From the Euclidean division we know that for any positive integer $l\in\Z_+$ there exists unique integers $q_i, r_i$, $i=1,2$, the quotient and remainder, such that
 \begin{align*}
   l&=q_ip_i+r_i, & 0&\leq r_i< p_i,& i&=1,2.
  \end{align*}
  After a  study of the shape of $J$ we can state
  \begin{lemma}
    For $l\geq \max\{|\vec n_1|,|\vec n_2|\}$ the  only nonzero elements of $J$ along
 a given row or column are
 \begin{align*}
  \begin{array}{ccccccccc}
              &   &       &   &         J_{(l-1)_{\{-,r_1(a_1(l-1)+1)\}},l}       &   &       &        &            \\
                                          &   &       &   &    *    &   &       &    &                                  \\
                                          &   &       &   & \vdots  &   &       &    &                                  \\
                                          &   &       &   &    *    &   &       &    &                                 \\
   J_{l,(l-1)_{\{-,r_2(a_2(l-1)-1)\}}} & * & \cdots & * & J_{l,l} & * & \cdots & *  & J_{l,(l+1)_{\{+,r_1(a_1(l+1)-1)\}}}\\
                                          &   &       &   &    *    &   &       &    &                                  \\
                                          &   &       &   & \vdots  &   &       &    &                                  \\
                                          &   &       &   &    *    &   &       &    &                                  \\
                 &   &       &   &         J_{(l+1)_{\{+,r_2(a_2(l+1)-1)\}},l}    &   &       &          &                            \\
  \end{array}
 \end{align*}
  \end{lemma}
Using this Lemma we get the desired result and the proof is complete.
\end{proof}

\paragraph{Example}
In order to be more clear let us suppose that $p_1=3$ and $p_2=2$ with  $\vec n_1=(4,3,2)$ and $\vec n_2=(3,2)$. The corresponding Jacobi type matrix has the following shape
\begin{align}\label{J}
\small
J=\left(\begin{BMAT}{c|c}{c|c}
        J^{[12]}& J^{[12,\geq 12]}\\
        J^{[\geq 12,12]} & J^{[\geq 12]}
        \end{BMAT} \right)
        =\left(
  \begin{BMAT}{cccccccccccc|ccccccccccccccc}{cccccccccccc|ccccccccccccccc}
   \textcolor{blue}{*} & \boldsymbol{\textcolor{blue}{1}} & 0 & 0 & 0 & 0 & 0 & 0 & 0 & 0 & 0 & 0 & 0 & 0 & 0 & 0 & 0 & 0 & 0 & 0 & 0 & 0 & 0 & 0 & 0 & 0 & \dots \\
 \boldsymbol{\textcolor{blue}{*}} & \boldsymbol{\textcolor{blue}{*}} &\boldsymbol{\textcolor{blue}{1}} & 0 & 0 & 0 & 0 & 0 & 0 & 0 & 0 & 0 & 0 & 0 & 0 & 0 & 0 & 0 & 0 & 0 & 0 & 0 & 0 & 0 & 0 & 0 & \dots \\
    0 & \boldsymbol{\textcolor{blue}{*}} & \boldsymbol{\textcolor{blue}{\textcolor{blue}{*}}} &\boldsymbol{\textcolor{blue}{1}} & 0 & 0 & 0 & 0 & 0 & 0 & 0 & 0 & 0 & 0 & 0 & 0 & 0 & 0 & 0 & 0 & 0 & 0 & 0 & 0 & 0 & 0 & \dots \\
    0 & 0 & \boldsymbol{\textcolor{blue}{*}} & \boldsymbol{\textcolor{blue}{*}} & \boldsymbol{\textcolor{blue}{*}} & \boldsymbol{\textcolor{blue}{*}} &\boldsymbol{ \textcolor{blue}{*}} & \boldsymbol{\textcolor{blue}{*}} & \boldsymbol{\textcolor{blue}{*}} &\boldsymbol{\textcolor{blue}{1}} & 0 & 0 & 0 & 0 & 0 & 0 & 0 & 0 & 0 & 0 & 0 & 0 & 0 & 0 & 0 & 0 & \dots \\
    0 & 0 & \boldsymbol{\textcolor{blue}{*}} & \textcolor{blue}{*} & \textcolor{blue}{*} & \textcolor{blue}{*} &\textcolor{blue}{*} & \textcolor{blue}{*} & \textcolor{blue}{*} & \boldsymbol{\textcolor{blue}{*}} & 0 & 0 & 0 & 0 & 0 & 0 & 0 & 0 & 0 & 0 & 0 & 0 & 0 & 0 & 0 & 0 & \dots \\
    0 & 0 & \boldsymbol{ \textcolor{blue}{*}} & \boldsymbol{\textcolor{blue}{*}} & \boldsymbol{\textcolor{blue}{*}} &\textcolor{blue}{*} & \textcolor{blue}{*} & \textcolor{blue}{*} &\textcolor{blue}{*} & \boldsymbol{\textcolor{blue}{*}} & 0 & 0 & 0 & 0 & 0 & 0 & 0 & 0 & 0 & 0 & 0 & 0 & 0 & 0 & 0 & 0 & \dots \\
    0 & 0 & 0 & 0 & \boldsymbol{\textcolor{blue}{*}} &\textcolor{blue}{*} & \textcolor{blue}{*} & \textcolor{blue}{*} &\textcolor{blue}{*} & \boldsymbol{\textcolor{blue}{*}} & \boldsymbol{\textcolor{blue}{*}} & \boldsymbol{\textcolor{blue}{*}} & \boldsymbol{\textcolor{blue}{*}} &\boldsymbol{ \textcolor{blue}{1}} & 0 & 0 & 0 & 0 & 0 & 0 & 0 & 0 & 0 & 0 & 0 & 0 & \dots \\
    0 & 0 & 0 & 0 & \boldsymbol{\textcolor{blue}{*}} & \textcolor{blue}{*} & \textcolor{blue}{*} & \textcolor{blue}{*} & \textcolor{blue}{*} &\textcolor{blue}{*} & \textcolor{blue}{*} & \textcolor{blue}{*} & \textcolor{blue}{*} & \boldsymbol{\textcolor{blue}{*}} & 0 & 0 & 0 & 0 & 0 & 0 & 0 & 0 & 0 & 0 & 0 & 0 & \dots \\
    0 & 0 & 0 & 0 & \boldsymbol{\textcolor{blue}{*}} & \boldsymbol{\textcolor{blue}{*}} & \boldsymbol{\textcolor{blue}{*}} & \boldsymbol{\textcolor{blue}{*}} & \textcolor{blue}{*} & \textcolor{blue}{*} &\textcolor{blue}{*} & \textcolor{blue}{*} & \textcolor{blue}{*} & \boldsymbol{\textcolor{blue}{*}} & \boldsymbol{\textcolor{blue}{*}} & \boldsymbol{\textcolor{blue}{*}} &\boldsymbol{\textcolor{blue}{1}} & 0 & 0 & 0 & 0 & 0 & 0 & 0 & 0 & 0 & \dots \\
    0 & 0 & 0 & 0 &   0 & 0 & 0 & \boldsymbol{\textcolor{blue}{*}} & \textcolor{blue}{*} & \textcolor{blue}{*} & \textcolor{blue}{*} &\textcolor{blue}{*}& \textcolor{blue}{*} & \textcolor{blue}{*} &\textcolor{blue}{*} & \textcolor{blue}{*} & \boldsymbol{\textcolor{blue}{*}} & 0 & 0 & 0 & 0 & 0 & 0 & 0 & 0 & 0 & \dots \\
    0 & 0 & 0 & 0 &   0 & 0 & 0& \boldsymbol{\textcolor{blue}{*}} & \boldsymbol{\textcolor{blue}{*}} & \boldsymbol{\textcolor{blue}{*}} &\textcolor{blue}{*}& \textcolor{blue}{*} & \textcolor{blue}{*} &\textcolor{blue}{*}& \textcolor{blue}{*} & \textcolor{blue}{*} & \boldsymbol{\textcolor{blue}{*}} & 0 & 0 & 0 & 0 & 0 & 0 & 0 & 0 & 0 & \dots \\
    0 & 0 & 0 & 0 & 0 & 0 & 0 & 0 & 0 & \boldsymbol{\textcolor{blue}{*}} & \textcolor{blue}{*} & \textcolor{blue}{*} & \textcolor{blue}{*} & \textcolor{blue}{*} &\textcolor{blue}{*}& \textcolor{blue}{*} & \boldsymbol{\textcolor{blue}{*}} & 0 & 0 & 0 & 0 & 0 & 0 & 0 & 0 & 0 & \dots \\
    0 & 0 & 0 & 0 & 0 & 0 & 0 & 0 & 0 & \boldsymbol{\textcolor{blue}{*}} &\textcolor{blue}{*} & \textcolor{blue}{*} &\textcolor{blue}{*}& \textcolor{blue}{*} & \textcolor{blue}{*} &\textcolor{blue}{*}& \boldsymbol{\textcolor{blue}{*}}  &\boldsymbol{ \textcolor{blue}{*}} &\boldsymbol{\textcolor{blue}{1}} & 0 & 0 & 0 & 0 & 0 & 0 & 0 & \dots \\
    0 & 0 & 0 & 0 & 0 & 0 & 0 & 0 & 0 & \boldsymbol{\textcolor{blue}{*}} &\boldsymbol{\textcolor{blue}{*}} & \boldsymbol{\textcolor{blue}{*}} & \boldsymbol{\textcolor{blue}{*}} &\textcolor{blue}{*}& \textcolor{blue}{*} & \textcolor{blue}{*} &\textcolor{blue}{*} & \textcolor{blue}{*} & \boldsymbol{\textcolor{blue}{*}} & 0 & 0 & 0 & 0 & 0 & 0 & 0 & \dots \\
    0 & 0 & 0 & 0 & 0 & 0 & 0 & 0 &  0 & 0 &0 & 0 & \boldsymbol{\textcolor{blue}{*}} &\textcolor{blue}{*} &\textcolor{blue}{*} & \textcolor{blue}{*} &\textcolor{blue}{*} & \textcolor{blue}{*} &\boldsymbol{ \textcolor{blue}{*}} & 0 & 0 & 0 & 0 & 0 & 0 & 0 & \dots \\
    0 & 0 & 0 & 0 & 0 & 0 & 0 & 0 & 0 & 0 & 0 & 0 &\boldsymbol{\textcolor{blue}{*}} & \boldsymbol{ \textcolor{blue}{*}} & \boldsymbol{\textcolor{blue}{*}} & \textcolor{blue}{*} & \textcolor{blue}{*}  & \textcolor{blue}{*} & \boldsymbol{\textcolor{blue}{*}}& \boldsymbol{\textcolor{blue}{*}} & \boldsymbol{\textcolor{blue}{*}} &\boldsymbol{\textcolor{blue}{*}} &\boldsymbol{\textcolor{blue}{1}} & 0 & 0 & 0 & \dots \\
    0 & 0 & 0 & 0 & 0 & 0 & 0 & 0 & 0 & 0 & 0 & 0 &0 & 0 & \boldsymbol{\textcolor{blue}{*}} & \textcolor{blue}{*} & \textcolor{blue}{*}  & \textcolor{blue}{*} & \textcolor{blue}{*}& \textcolor{blue}{*} & \textcolor{blue}{*} &\textcolor{blue}{*} & \boldsymbol{\textcolor{blue}{*}} & 0 & 0 & 0 & \dots \\
    0 & 0 & 0 & 0 & 0 & 0 & 0 & 0 & 0 & 0 & 0 & 0 &0 & 0 & \boldsymbol{\textcolor{blue}{*}} & \textcolor{blue}{*} & \textcolor{blue}{*}  & \textcolor{blue}{*} & \textcolor{blue}{*}& \textcolor{blue}{*} & \textcolor{blue}{*} &\textcolor{blue}{*} & \boldsymbol{\textcolor{blue}{*}} & \boldsymbol{\textcolor{blue}{*}} & \boldsymbol{\textcolor{blue}{*}} &\boldsymbol{\textcolor{blue}{1}} & \dots \\
    0 & 0 & 0 & 0 & 0 & 0 & 0 & 0 & 0 & 0 & 0 & 0 & 0 & 0 &\boldsymbol{\textcolor{blue}{*}} & \boldsymbol{\textcolor{blue}{*}} & \boldsymbol{\textcolor{blue}{*}}  & \boldsymbol{\textcolor{blue}{*}} & \textcolor{blue}{*}& \textcolor{blue}{*} & \textcolor{blue}{*} &\textcolor{blue}{*} & \textcolor{blue}{*} & \textcolor{blue}{*} & \textcolor{blue}{*} &\boldsymbol{ \textcolor{blue}{*}} & \dots \\
    0 & 0 & 0 & 0 & 0 & 0 & 0 & 0 & 0 & 0 & 0 & 0 & 0 & 0 & 0 & 0 & 0 & \boldsymbol{\textcolor{blue}{*}} & \textcolor{blue}{*}& \textcolor{blue}{*} & \textcolor{blue}{*} &\textcolor{blue}{*} & \textcolor{blue}{*} & \textcolor{blue}{*} & \textcolor{blue}{*} & \boldsymbol{\textcolor{blue}{*}} &  \dots \\
    0 & 0 & 0 & 0 & 0 & 0 & 0 & 0 & 0 & 0 & 0 & 0 & 0 & 0 & 0 & 0 & 0 & \boldsymbol{\textcolor{blue}{*}} & \boldsymbol{\textcolor{blue}{*}}& \boldsymbol{\textcolor{blue}{*}} & \textcolor{blue}{*} &\textcolor{blue}{*} & \textcolor{blue}{*} & \textcolor{blue}{*} & \textcolor{blue}{*} & \boldsymbol{\textcolor{blue}{*}}  & \dots \\
    0 & 0 & 0 & 0 & 0 & 0 & 0 & 0 & 0 & 0 & 0 & 0 & 0 & 0 & 0 & 0 & 0 & 0 & 0 &  \boldsymbol{\textcolor{blue}{*}} & \textcolor{blue}{*} &\textcolor{blue}{*} & \textcolor{blue}{*} & \textcolor{blue}{*} & \textcolor{blue}{*} & \boldsymbol{\textcolor{blue}{*}}   & \dots \\
    0 & 0 & 0 & 0 & 0 & 0 & 0 & 0 & 0 & 0 & 0 & 0 & 0 & 0 & 0 & 0 & 0 & 0 & 0 & \boldsymbol{\textcolor{blue}{*}} & \textcolor{blue}{*} &\textcolor{blue}{*} & \textcolor{blue}{*} & \textcolor{blue}{*} & \textcolor{blue}{*} &\textcolor{blue}{*} &\dots\\
0 & 0 & 0 & 0 & 0 & 0 & 0 & 0 & 0 & 0 & 0 & 0 & 0 & 0 & 0 & 0 & 0 & 0 & 0 & \boldsymbol{\textcolor{blue}{*}} & \boldsymbol{\textcolor{blue}{*}} &\boldsymbol{\textcolor{blue}{*}} & \boldsymbol{\textcolor{blue}{*}} & \textcolor{blue}{*} & \textcolor{blue}{*} & \textcolor{blue}{*}  & \dots \\
    0 & 0 & 0 & 0 & 0 & 0 & 0 & 0 & 0 & 0 & 0 & 0 & 0 & 0 & 0 & 0 & 0 & 0 & 0 & 0 & 0 & 0 & \boldsymbol{\textcolor{blue}{*}} & \textcolor{blue}{*} & \textcolor{blue}{*} & \textcolor{blue}{*} & \dots\\
  0 & 0 & 0 & 0 & 0 & 0 & 0 & 0 & 0 & 0 & 0 & 0 & 0 & 0 & 0 & 0 & 0 & 0 & 0 &0 &0 &0 & \boldsymbol{\textcolor{blue}{*}} & \boldsymbol{\textcolor{blue}{*}} & \boldsymbol{\textcolor{blue}{*}} & \textcolor{blue}{*}  & \dots \\
    \vdots & \vdots & \vdots & \vdots & \vdots & \vdots & \vdots & \vdots & \vdots & \vdots & \vdots & \vdots & \vdots & \vdots & \vdots & \vdots & \vdots & \vdots & \vdots & \vdots & \vdots & \vdots & \vdots & \vdots & \vdots & \vdots & \ddots \\
  \end{BMAT}
\right),
\end{align}
where $*$ denotes a non-necessarily null real number.

In our example ($p_1=3$, $p_2=2$, $\vec n_1=(4,3,2)$ and $\vec n_2=(3,2)$)  for $[l]=[12]$ we have
\begin{align*}
 (y-x)K^{[12]}(x,y)&=\left[ \sum_{i=8}^{11} \sum_{j=12}^{15} \bar{Q}(x)^{(i)}J_{i,j} Q(y)^{(j)}\right]-
 \left[\sum_{i=12}^{13} \sum_{j=9}^{11} \bar{Q}(x)^{(i)}J_{i,j} Q(y)^{(j)}\right]
\end{align*}

\subsection{Expressing the Jacobi type matrix in terms of factorization factors}
As we have seen we can write $J$ in terms of $S$ or of $\bar S$, this means that for each term of $J$ has two different expressions, giving relations between
$S$ with $\bar S$. We are not too concerned about these relations since what we want here is the most simple expression we can get
for the elements of $J$. It is easy to realize that this is achieved if we use the expression involving $S$ in order
to calculate the upper part of $J$ and the expression involving $\bar S$ to calculate the lower part of it.  Hence, for every $J_{l,k}$ we will have expressions in terms of the
factorization matrices coefficients and the elements of their inverses --thus, in terms of the MOPRL and associated second kind functions.  The only terms from the factorization matrices (or their inverses) that will be
involved when calculating any $J_{l,k}$ are just those between the main diagonal  and the $l-|n_{1}|$ diagonal (both included)
of $S$ and those between the main diagonal and the $l+|n_2|$ diagonal (both included) of $\bar S$. And not even all of them.
As we are about to see there are three different kinds of elements in $J$. The ones
along the main diagonal, the ones along the immediate closest diagonals to the main one, and finally all the remaining diagonals.
The recursion relation coefficients $J_{k,l}$  are
 ultimately related to the MOPRL and its associated  second kind functions in the following way

\begin{pro}
 The elements of the recursion matrix  $J$ can be written in terms of products of the entries of the LU factorization matrices and its inverses as follows
 \begin{align*}
& \begin{aligned}
   J_{l,l}&=S_{l,(l-1)_{\{-,a_1(l)\}}}+S^{-1}_{(l+1)_{\{+,a_1(l)\}},l}+\sum_{\substack{a=1,\dots,p_1\\a\neq a_{1}(l)}} S_{l,(l-1)_{\{-,a\}}} S^{-1}_{(l+1)_{\{+,a\}},l},\\
  &= \bar S_{l,(l+1)_{\{+,a_2(l)\}}}\bar S^{-1}_{l,l}+\bar S_{l,l}\bar S^{-1}_{(l-1)_{\{-,a_2(l)\}},l}+ \sum_{\substack{a=1,\dots,p_2\\a\neq a_{2}(l)}} \bar S_{l,(l+1)_{\{+,a\}}} \bar S^{-1}_{(l-1)_{\{-,a\}},l},
 \end{aligned}
\\
&\begin{aligned}
 J_{l,l+1}&=S^{-1}_{(l+1)_{\{+,a_1(l)\}},l+1}+ \sum_{\substack{a=1,\dots,p_1\\a\neq a_{1}(l)}} S_{l,(l-1)_{\{-,a\}}} S^{-1}_{(l+1)_{\{+,a\}},l+1}, \\
  J_{l+1,l}&=   \bar S_{l+1,(l+1)_{\{+,a_2(l)\}}} \bar S^{-1}_{l,l} +\sum_{\substack{a=1,\dots,p_2\\a\neq a_{2}(l)}} \bar S_{l+1,(l+1)_{\{+,a\}}} \bar S^{-1}_{(l-1)_{\{-,a\}},l},
\end{aligned}\\
&\begin{aligned}
   J_{l,l+k}&=\sideset{}{^{p_1}}\sum_{a=r_1(a_1(l+k-1)+1)}^{r_1(a_1(l)-1)}
 S_{l,(l-1)_{\{-,a_1\}}} S^{-1}_{(l+1)_{\{+,a_1\}},l+k} &
  2&\leq k\leq (l+1)_{\{+,r_1(a_1(l+1)-1)\}}-l, \\
  J_{l+k,l}&=\sideset{}{^{p_2}}\sum_{a=r_2(a_2(l+k-1)+1)}^{r_1(a_2(l)-1)}
 \bar S_{l+k,(l+1)_{\{+,a_\}}} \bar S^{-1}_{(l-1)_{\{-,a\}},l}, &  2&\leq k\leq (l+1)_{\{+,r_2(a_2(l+1)-1)\}}-l.
\end{aligned}
 \end{align*}
\end{pro}
Where, for $r,r'<p$, we have used
\begin{align*}
\sideset{}{^p}\sum_{a=r}^{r'}X_a=\begin{dcases}
  \sum_{a=r}^{r'} X_a, & r\leq r',\\
  \sum_{a=1}^{r'}X_a+\sum_{a=r}^pX_a, & r>r'.
\end{dcases}
\end{align*}

\section*{Acknowledgements}
GA thanks economical support from the Universidad Complutense de Madrid  Program ``Ayudas para Becas y Contratos Complutenses Predoctorales en Espa\~{n}a 2011". MM thanks economical support from the Spanish ``Ministerio de Econom\'{\i}a y Competitividad" research project MTM2012-36732-C03-01,  \emph{Ortogonalidad y aproximacion; Teoria y Aplicaciones}.

\end{document}